\documentclass[10pt,draft,reqno]{amsart}
     \makeatletter
     \def\section{\@startsection{section}{1}%
     \z@{.7\linespacing\@plus\linespacing}{.5\linespacing}%
     {\bfseries
     \centering
     }}
     \def\@secnumfont{\bfseries}
     \makeatother
\newtheorem{theorem}{Theorem}[section]
\newtheorem{lemma}[theorem]{Lemma}

\theoremstyle{definition}
\newtheorem{definition}[theorem]{Definition}

\theoremstyle{remark}

\numberwithin{equation}{section}
\begin{document}

\title[Stochastic control for BSDEs with semi-Markov chain noises]{Stochastic control for Backward Stochastic Differential Equations with semi-Markov chain noises}

\author{Robert J. Elliott}
\address{Robert J. Elliott: Haskayne School of Business, University of Calgary, Calgary, AB, T2N 1N4, Canada; Business, University of South Australia, Adelaide, Australia 5001}
\email{relliott@ucalgary.ca}

\author[Zhe Yang]{Zhe Yang}
\address{Zhe Yang: Haskayne School of Business, University of Calgary, Calgary, AB, T2N 1N4, Canada}
\email{yangzhezhe@gmail.com}

\subjclass[2010] {Primary 60K15; Secondary 60J10}

\keywords{BSDEs,  discrete-time, Markov chain models, semi-Markov chain models}

\begin{abstract}
In this paper, we extend the results of Elliott and Yang
\cite{elliott3} and discuss the control of a stochastic process for which
the driving noise is provided by a martingale associated with a semi-Markov
Chain. An existence and a comparison theorem are obtained. In our discrete
time setting, adjoint processes are provided by backward stochastic difference
equations. Technical results from partial differential equation theory
to establish a verification theorem are not required.
\end{abstract}

\maketitle

\section{Introduction}

In classical control problems the noise is usually Gaussian. However, in discrete time many basic models employ Markov chains. The canonical representation for a Markov process is as the sum of a predictable process and a martingale. See, for example, Elliott, Aggoun and Moore \cite{Aggoun}. This martingale represents the random noise, which in general is not Gaussian. \\
\indent The usual definition of a Markov process is that, given the present, the future is independent of the past, roughly, that the probability of where the process goes next depends only on where it is now, rather than on its past history. However, it is not often recognized that this property imposes restrictions on the probabilities of how long the process stays in any state. For a finite state Markov process in continuous time its occupation times in any state must be described by an exponential random variable; for such a process in discrete time the occupation times, or sojourn times, must be described by a geometric random variable. Empirically, processes with more general sojourn times are observed. They are called semi-Markov processes.\\
\indent In previous papers, without loss of generality, the state space of a finite state process can be identified with the standard unit vectors in a Euclidean space. In our paper Elliott and Malcolm \cite{elliott} it is shown how, in discrete time, the state space for a semi-Markov process can be identified with unit vectors in countably many copies of a Euclidean space. This provides an explicit   semimartingale representation for the semi-Markov chain and its associated noise process.\\
\indent A novel control problem with this semi martingale noise is discussed in this paper and a characterization of the optimal control provided.\\
\indent The control of diffusion processes has been extensively studied. See, for example, the
books of Elliott \cite{elliott1982} and Yong and Zhou\cite{Yong}. In 1997 El Karoui, Peng and Quenez \cite{Karoui} obtained a duality between linear
stochastic differential equations (SDEs) and linear backward stochastic differential
equations (BSDEs) driven by Brownian motion. That is, the solutions of the BSDEs
were described by solutions of SDEs driven by Brownian motion. They applied this
duality to stochastic control problems. Using duality, Yang and Elliott \cite{yang} in 2020 solved the analogous stochastic control problem for BSDEs and anticipated BSDEs with Markov chain noises.\\
\indent  Elliott and Malcolm \cite{elliott} in 2021 derived
semimartingale dynamics for a semi-Markov chain and gave them in a new
vector form which explicitly exhibits the times at which jump-events occur
and the probabilities of state transitions. A useful result in Elliott and Malcolm \cite{elliott}
is a new vector lattice state-space representation for a general finite-state,
discrete-time semi-Markov chain. On this space the semi-Markov chain and
its occupation times are a Markov process with dynamics described by transition
matrices. Following the framework of Elliott and Malcolm \cite{elliott}, Elliott and Yang \cite{elliott3} discussed BSDEs in a semi-Markov Chain Model and extended the results of Cohen and Elliott \cite{Cohen} to semi-Markov chains. They found the solution $Z_i$ took values in spaces whose dimension is increasing as the time $i$ increases.\\
\indent This paper, using duality, extends the control results of Yang and Elliott \cite{yang} to BSDEs with semi-Markov chain noise.\\
\indent The present paper is structured as follows: Section 2 presents the model and some preliminary results. Section 3 gives results for BSDEs in a semi-Markov chain model. Section 4 gives a comparison theorem and Section 5 presents the stochastic control results which characterize the optional control, for BSDEs with semi-Markov chain noise with the help of duality.

\section{The  Model and Some Preliminary Results}
In this sction we review the definitions and representation of a finite state semi-Markov chain. This generalizes, in a natural way, a Markov chain and is the process discussed in this paper.\\
\indent Consider a discrete time, finite state space semi-Markov chain $X_.=\{X_k,~$ $k=0,1,2,... \}$ defined on a
probability space $(\Omega,\mathcal{F},P)$. If the state space has $N\in\mathbb{N}$ elements, they can be identified with the set of unit vectors $S=\{e_1,e_2\cdots,e_N\}$, where $e_i=(0,\cdots 0,1,0, \cdots,0) '$ $\in
\mathbb{R}^N$ with 1 in the $i$-th position. Suppose the initial state $X_0\in S$ is given, or its probability distribution $p_0=(p_0^1,p_0^2,...,p_0^N)\in[0,1]^N$ is known. The chain will change state at random discrete times $\tau_n$. State transitions at these times are of the type $e_i\rightarrow e_j$, with $i\neq j.$ Set $\tau_0:=0$. Then we know successive jump event times form a strictly increasing sequence, that is, $\tau_0<\tau_1<\tau_2<...$. Set $\mathcal{F}_m:=\sigma\{X_k,~k\leq m\}$.\\
\indent Now we give the definition of a time-homogeneous semi-Markov chain following Elliott and Malcolm \cite{elliott}.
\begin{definition}\label{smc}
The stochastic process $X_.$ is a time-homogeneous semi-Markov chain if
\begin{align*}
&P(X_{\tau_{n+1}}=e_j,\tau_{n+1}-\tau_{n}=m \mid \mathcal{F}_{\tau_n})\\[2mm]
=&P(X_{\tau_{n+1}}=e_j,\tau_{n+1}-\tau_{n}=m \mid X_{\tau_{n}}=e_i).
\end{align*}
We write this as $q(e_j,e_i,m)$.
\end{definition}
Set
\begin{align*}
f_{j,i}(m)&:=P(\tau_{n+1}-\tau_{n}=m \mid X_{\tau_{n+1}}=e_j, X_{\tau_{n}}=e_i)~~~~\text{and}\\[2mm]
p_{j,i}&:=P( X_{\tau_{n+1}}=e_j| X_{\tau_{n}}=e_i).
\end{align*}
Consequently
\begin{align*}
&q(e_j,e_i,m)\\[2mm]
=&P(\tau_{n+1}-\tau_{n}=m \mid X_{\tau_{n+1}}=e_j, X_{\tau_{n}}=e_i)P( X_{\tau_{n+1}}=e_j| X_{\tau_{n}}=e_i)\\[2mm]
=&f_{j,i}(m)p_{j,i}.\end{align*}
We can also consider the factorization: Set
\begin{align*}
\pi_{i}(m)&:=P( \tau_{n+1}-\tau_n=m\mid X_{\tau_{n}}=e_i)~~~~\text{and}\\[2mm]
p_{j,i}(m)&:=P(X_{\tau_{n+1}}=e_j \mid \tau_{n+1}-\tau_{n}=m, X_{\tau_{n}}=e_i).
\end{align*}
Thus,
\begin{align*}
&q(e_j,e_i,m)\\[2mm]
=&P( \tau_{n+1}-\tau_n=m\mid X_{\tau_{n}}=e_i)P(X_{\tau_{n+1}}=e_j \mid \tau_{n+1}-\tau_{n}=m, X_{\tau_{n}}=e_i)\\[2mm]
=&\pi_{i}(m)p_{j,i}(m).
\end{align*}
Write
\begin{align*}
G_{i}(m)&:=P( \tau_{n+1}-\tau_n\leq m\mid X_{\tau_{n}}=e_i)=\sum^m_{l=1}\pi_i(l),\\[2mm]
F_i(m)&:=P( \tau_{n+1}-\tau_n> m\mid X_{\tau_{n}}=e_i)=1-G_{i}(m).
\end{align*}
Denote by $\Delta^i(m)$ the conditional probability for a state-transition to occur at the next discrete time, that is,
$$\Delta^i(m):=\frac{\pi_{i}(m)}{F_i(m-1)}.$$
\begin{definition}\label{cumulative}
For each index $i$, $1 \leq i \leq N$, we define the recursive process
$h^i_k := \langle X_k, e_i\rangle+ \langle X_k, e_i\rangle\langle X_k, X_{k-1}\rangle h^i_{k-1}$, with $h^i
_0 := \langle
X_0, e_i
\rangle \in \{0, 1\}$. Thus, the $h^i_.$
processes are non-zero only at times when $X_.= e_i $. The process $h^i_.$ returns the cumulative time spent in state $e_i$.\\
If $h_k =\sum\limits^{N}\limits_{i=1}h^i_k$, then $h_0 = 1$ and $h_k = 1+\langle X_k, X_{k-1}\rangle h_{k-1}$. The process $h_k$ measures
the amount of time since the last transition event. This process is never zero.
\end{definition}
For $m=1,2,...$, denote $A(m)$ for the $N\times N$ matrix with entries $a_{i,i}(m)=1-\Delta^{i}(m)$ and $a_{j,i}(m)=p_{j,i}(m)\Delta^{i}(m).$\\[2mm]
\indent The following lemma comes from Elliott and Malcolm \cite{elliott}.
\begin{lemma}\label{representation of SMC}
The semi-Markov chain $X_.$ has the following semi-martingale dynamics:
$$
X_{k+1}=A(h_k)X_k+M_{k+1}\in\mathbb{R}^N.
$$
Here $M_{k+1}$ is a martingale increment: $E[M_{k+1}|X_k,h_k]=0\in\mathbb{R}^N.$ Here, $E$ denotes expectation.
\end{lemma}
Now recall Section 3 of Elliott and Malcolm \cite{elliott}. This provides a representation of the process $(X_.,h_.)$ as a Markov chain on a lattice. \\
\indent The state of the chain $X_k\in\{e_1,e_2,...,e_N\}$ and the number of time steps $h_k$ describe the semi-Markov chain $X_.$. Then a state space $\bar{S}$ for the chain $\bar{X}_k:=(X_k,h_k),~k\in\{1,2,...\}$ can be identified with countably many copies of $S$ as follows: Elements of $\bar{S}$ can be thought of as infinite column vectors:
\begin{align*}
(e_{l,i})~~~\mbox{corresponds~to}~~~(0,0,...,0|\cdot\cdot~\cdot&|\underbrace{0,...,1,...,0}|\cdot\cdot\cdot|0,0,...,0|\cdot\cdot~\cdot)',\\
&\quad\quad h_k=i
\end{align*}
with $e_l=(0,...,1,...,0)'$ in the $i^{th}$ block, where $l=1,2,...,N,$ and $i=0,1,...$. As a basis of unit vectors for the process $\bar{X}_.=(X_.,h_.)$, we take the unit vectors $e_{l,i}$ with $l=1,2,...,N,$ and $i=0,1,...$. Here $l$ denotes the state $e_l$ in
$\{e_1,e_2,...,e_N\}$ and the $i$ corresponds to the sojourn time in the state $e_l$ since the last jump. Write $\bar{S}=\{e_{l,i};~l=1,2,...,N,$ and $i=0,1,...\}$ with $e_{l,i}$ representing $e_{l}$ in the $i^{th}$ block. For any $m\in\{1,2,...\}$, write
\begin{align*}
\Pi(m)  =
\begin{pmatrix}
0&p_{1,2}(m)\Delta^2(m)&\cdot\cdot\cdot&p_{1,N}(m)\Delta^N(m)\\[2mm]
p_{2,1}(m)\Delta^1(m)&0&\cdot\cdot\cdot&p_{2,N}(m)\Delta^N(m)\\[2mm]
\cdot\cdot\cdot\\[2mm]
p_{N,1}(m)\Delta^1(m)&p_{N,2}(m)\Delta^2(m)&\cdot\cdot\cdot&0
\end{pmatrix}
\end{align*}
and $D(m)=$\text{diag}$\{1-\Delta^1(m),1-\Delta^2(m),...,1-\Delta^N(m)\}$. With $0$ representing the $N\times N$ zero matrix write
\begin{align*}
C  =
\begin{pmatrix}
\Pi(1)&\Pi(2)&\Pi(3)&\cdot\cdot\cdot&\Pi(T)&\cdot\cdot\cdot\\[2mm]
D(1)&0&0&\cdot\cdot\cdot&0&\cdot\cdot\cdot\\[2mm]
0&D(2)&0&\cdot\cdot\cdot&0&\cdot\cdot\cdot\\[2mm]
\cdot\cdot\cdot\\[2mm]
0&0&0&\cdot\cdot\cdot&D(T)&\cdot\cdot\cdot\\[2mm]
\cdot\cdot\cdot
\end{pmatrix}.
\end{align*}
Write the enlarged vectors as $\bar{X}_i.$ Then, following Elliott and Malcolm \cite{elliott}, the semi-martingale dynamics of Markov chain can be written
$$\bar{X}_{i+1}=C\bar{X}_i+\bar{M}_{i+1}\in \bar{S}$$
with $E[\bar{M}_{i+1}|\bar{X}_i]=E[\bar{X}_{i+1}|\bar{X}_i]-C\bar{X}_i=0$. When the  time horizon is $T<+\infty$,
the matrix $C$ is finite, and the state space of $\bar{X}_.$ at time $i$ only has $(i+1)N$ elements.
\begin{definition}\label{seminorm}
For any integer $K$, and any $i\in\{0,1,...,T-1\},$ we shall denote by $\|\cdot\|_{\bar{M}_{i+1}}$ the seminorm on
the space of $\sigma(\bar{X}_{i})$-measurable and $\mathbb{R}^{K\times((i+1)N)}$-valued random variables $Z_i$, given by
\begin{align*}
\|Z_i\|^2_{\bar{M}_{i+1}} :&= E\text{Tr}
[\sum_{0\leq u\leq i}Z_{u}\cdot E[\bar{M}_{u+1}\bar{M}'_{u+1}|\bar{X}_u]\cdot Z'_u]\\
&=\sum_{0\leq u\leq i}\text{Tr}  E[(Z_u\bar{M}_{u+1})(Z_u\bar{M}_{u+1})'
].
\end{align*}
Here, Tr denotes the trace of a matrix.
\end{definition}

\begin{lemma}\label{equivalent} For any $i\in\{0,1,...,T-1\}$, the following statements are equivalent:\\[2mm]
$(i)~\|Z_i^{(1)}-Z_i^{(2)}\|^2_{\bar{M}_{i+1}}=0.$\\[2mm]
$(ii)~E\text{Tr} [(Z_i^{(1)}-Z_i^{(2)})\bar{M}_{i+1}\bar{M}'_{i+1}(Z_i^{(1)}-Z_i^{(2)})']=0.$\\[2mm]
$(iii)~Z_i^{(1)}\bar{M}_{i+1}=Z_i^{(2)}\bar{M}_{i+1},$ $\mathbb{P}$-a.s.\\[2mm]
$(iv)~\sum\limits_{0\leq u\leq i}Z_u^{(1)}\bar{M}_{u+1}=\sum\limits_{0\leq u\leq i}Z_u^{(2)}\bar{M}_{u+1},$ $\mathbb{P}$-a.s.\\[2mm]
In this case we shall write $Z_i^{(1)}\sim_{\bar{M}_{i+1}}Z_i^{(2)}$.
\end{lemma}
Proof. See Lemma 3.2 in Elliott and Yang \cite{elliott3}.
\begin{definition}\label{equivalence}
For any two $\sigma(\bar{X}_{i-1})$-measurable random variables $Z_{i-1}^{(1)}$ and $Z_{i-1}^{(2)}$, we
shall write $Z_{i-1}^{(1)}\sim _{\bar{M}_i} Z_{i-1}^{(2)}$
if $Z_{i-1}^{(1)}\bar{M}_i = Z_{i-1}^{(2)}\bar{M}_i,$ $\mathbb{P}$-a.s.
\end{definition}

\section{Duality Results}
This section gives some results fo BSDEs. For $n\in\mathbb{N}$, and $\phi \in \mathbb{R}^n$, denote the Euclidean norm $|\phi|_{n}=\sqrt{\phi'\phi}$, and for $\psi \in \mathbb{R}^{n\times n}$, the matrix norm $\|\psi\|_{n\times n}=\sqrt{Tr(\psi'\psi)}$. Set
\begin{equation}\label{Psi}
\Psi_k:=\text{diag}(C\bar{X}_k)- \text{diag}(\bar{X}_k)C' - C \text{diag}(\bar{X}_k).
\end{equation}
Define the semi-norm $\|.\|_{\bar{X}_k}$, for
$C\in \mathbb{R}^{N(k+1)\times 1}$ as :
\begin{align*}
\|C\|^2_{\bar{X}_k} & = Tr(C' \Psi_kC).
\end{align*}
See Campbell and Meyer \cite{campbell} for the following definition:
\begin{definition}[Moore-Penrose pseudoinverse]\label{defMoore}
The Moore-Penrose pseudoinverse of a square matrix $Q$ is the matrix $Q^{\dagger}$ satisfying the properties:\\[2mm]
  1) $QQ^{\dagger}Q = Q$ \\[2mm]
  2) $Q^{\dagger}QQ^{\dagger} = Q^{\dagger}$ \\[2mm]
  3) $(QQ^{\dagger})' = QQ^{\dagger}$ \\[2mm]
  4) $(Q^{\dagger}Q)'=Q^{\dagger}Q.$
\end{definition}
\indent Although the dimension of $\bar{M}_i$ is increasing with $i$, the duality between the solutions to linear BSDEs
and linear SDEs, adapted to a one-dimensional time-discrete case with semi-Markov
chain noise, still holds by Theorem 2 of Cohen and Elliott \cite{Sam2} and Theorem 3.2 of Yang and Elliott \cite{yang}. The following four lemmas are proved in Yang and Elliott \cite{yang}.
\begin{lemma}\label{Duality1} (Duality) Let $(\alpha_.,\beta_.,g_.)$ be a $du\times P$-a.s. bounded $(\mathbb{R},$ $\mathbb{R}^{1\times (N(i+1))},$ $\mathbb{R})$-valued
adapted process for any $i\in\{0,1,...,T-1\},$ and $\xi$ be an essentially bounded, valued in $\mathbb{R}$, $\sigma(\bar{X}_T)$ -measurable random variable. $\Psi_.$ is given by equation (\ref{Psi}). Then the linear BSDE given by
\begin{align*}
\begin{cases}
Y_i= \xi + \sum\limits_{k=i}^{T-1} (\alpha_kY_k+ \beta_k \Psi_k^{\dagger}\Psi_kZ'_k+g_k)  -  \sum\limits_{k=i}^{T-1}Z_{k} \bar{M}_{k+1},\\
~~~~~~~~~~~~~~~~~~~~~~~~~~~~~~~~~~~~~~~~~~~~i\in\{0,1,...,T-1\};\\
Y_T= \xi
\end{cases}
\end{align*}
has an adapted solution $(Y_., Z_.)$. Here for any $i\in\{0,1,...,T-1\}$, we have $(Y_i,Z_i)$ is a $\sigma(\bar{X}_{i})$-measurable and $\mathbb{R}\times\mathbb{R}^{1\times((i+1)N)}$-valued random variable, (up to appropriate sets of
measure zero). Moreover, $Y_.$ is given by the explicit
formula
\begin{align*}
Y_i = E[\xi V_T +\sum_{k=i}^{T-1}g_kV_k|\bar{X}_i], ~~i\in\{0,1,...,T-1\}
\end{align*}
up to indistinguishability. Further, $V_.$ is the solution of the following 1-dimensional SDE:
\begin{equation}\label{Duality}
\begin{cases}
V_m=1+ \sum\limits_{k=i+1}^{m}\alpha_k V_k + \sum\limits_{k=i+1}^{m}V_{k}\beta_k\Psi_k^{\dagger}\Psi_k(\Psi_k^{\dagger})'\bar{M}_k,~~  m \in \{i+1,...,T\};\\[2mm]
V_i =1.
\end{cases}
\end{equation}
\end{lemma}
Proof. See Theorem 3.2 of Yang and Elliott \cite{yang}.
\begin{lemma}\label{sdebound}
We make the same assumptions as in Lemma \ref{Duality1}. Then we know there exists a constant $C>0$ such that the solution $V_.$ of SDE (\ref{Duality}) satisfies $E[\max\limits_{i=0,1,...,T}|V_i|^2]<C$.
\end{lemma}
Proof. See Lemma 3.3 of Yang and Elliott \cite{yang}.
\begin{lemma}\label{sde0} We make the same assumptions as in Lemma \ref{Duality1}. Denote
the bounds of $|\alpha_.|$ and $|\beta_.|_N$ by $p$ and $l$, respectively. If the constant $l$ satisfies for any $i\in\{0,1,...,T\},$
\begin{equation}\label{lconstant}
\sqrt{2}l\|\Psi_i\|_{N\times N} \|\Psi^{\dagger}_i\|_{N\times N}^2\leq 1,
\end{equation}
 then we know the solution $V_.$ of SDE (\ref{Duality}) satisfies for any $i\in\{0,1,...,T\}$, $V_i\geq 0.$
\end{lemma}
Proof. See Lemma 3.5 of Yang and Elliott \cite{yang}.
\begin{lemma}\label{normbound1} For any $k=0,1,...,T-1$,
there exists a constant $\lambda_k>0$ such that for any $B\in\mathbb{R}^{N(k+1)}$,
$$|\Psi_k^{\dagger}\Psi_kB|_{N(k+1)}\leq\lambda_k\|B\|_{\bar{X}_k}.$$
\end{lemma}
Proof. See Lemma 3.1 of Yang and Elliott \cite{yang}.\\
\indent Set $\lambda:=\max\{\lambda_0,\lambda_1,...\lambda_{T-1}\}$. From Lemma \ref{normbound1}, we deduce:
\begin{lemma}\label{normbound2} There exists a constant $\lambda>0$ such that for any $k=0,$ $1,...,T-1$, $B\in\mathbb{R}^{N(k+1)}$,
$$|\Psi_k^{\dagger}\Psi_kB|_{N(k+1)}\leq\lambda\|B\|_{\bar{X}_k}.$$
\end{lemma}
\indent We now quote the following lemma establishing the existence and uniqueness of a solution to a BSDE in a semi-Markov chain model from Elliott and Yang \cite{elliott3}.
\begin{lemma}\label{exisunique} Suppose $f$ is such that the following two assumptions hold:\\
(i) For any $i\in\{0,1,...,T-1\}$ and $Y_.$, if $Z^{(1)}_i\sim_{\bar{M}_{i+1}} Z^{(2)}_i$, then $f(i, Y_i, Z^{(1)}_i) = f(i, Y_i, Z^{(2)}_i)$, $\mathbb{P}$-a.s. \\
(ii) For any $i\in\{0,1,...,T-1\}$ and $Z_.$, the map $Y_i \mapsto Y_i -f(i, Y_i, Z_i)$ is $\mathbb{P}$-a.s. a bijection
from $\mathbb{R} \rightarrow\mathbb{R}$, up to equality $\mathbb{P}$-a.s.\\
Then for any terminal condition $Q$ essentially bounded, $\sigma(\bar{X}_T)$ -measurable, and
with values in $\mathbb{R}$, BSDE
\begin{align*}
\begin{cases}
Y_i= Q + \sum\limits_{u=i}^{T-1} f(u, Y_u, Z_u)  -  \sum\limits_{u=i}^{T-1}Z_{u} \bar{M}_{u+1},~~i\in\{0,1,...,T-1\};\\
Y_T= Q
\end{cases}
\end{align*}
has an adapted solution $(Y_., Z_.)$. Here for any $i\in\{0,1,...,T-1\}$, we have $(Y_i,Z_i)$ is a $\sigma(\bar{X}_{i})$-measurable and $\mathbb{R}\times\mathbb{R}^{1\times((i+1)N)}$-valued random variable. Moreover,
this solution is unique up to indistinguishability for $Y_.$ and equivalence $\sim _{\bar{M}_{i+1}}$ for
$Z_i$.
\end{lemma}
\section{A comparison result}
In this section we establish a comparison theorem.\\
\indent A comparison theorem for one-dimensional BSDEs with Markov chain noise in Yang, Ramarimbahoaka and Elliott \cite{Yang} (Theorem 3.2) extends to our one-dimensional time-discrete case with semi-Markov
chain noise: Let $(Y^{(1)}_.,Z^{(1)}_.)$ and $(Y^{(2)}_.,Z^{(2)}_.)$ be the solutions of the following two BSDEs with semi-Markov chain noise, respectively:
\begin{align*}
\begin{cases}
Y^{(1)}_i= Q_{1} + \sum\limits_{u=i}^{T-1} f_1(u, Y^{(1)}_u, Z^{(1)}_u)  -  \sum\limits_{u=i}^{T-1}Z^{(1)}_{u} \bar{M}_{u+1}, ~~i\in\{0,1,...,T-1\};\\
Y_T^{(1)}=Q_{1}
\end{cases}
\end{align*}
and
\begin{align*}
\begin{cases}
Y^{(2)}_i= Q_{2} + \sum\limits_{u=i}^{T-1} f_2(u, Y^{(2)}_u, Z^{(2)}_u)  -  \sum\limits_{u=i}^{T-1}Z^{(2)}_{u} \bar{M}_{u+1},~~i\in\{0,1,...,T-1\};\\
Y_T^{(2)}= Q_{2}.
\end{cases}
\end{align*}
\begin{lemma}\label{Comparison} Assume $Q_1,Q_2$ are essentially bounded, $\sigma(\bar{X}_T)$ -measurable, and
with values in $\mathbb{R}$, and $f_1,f_2$ satisfy conditions such that the above BSDEs have unique solutions.
Suppose the following conditions hold:\\[2mm]
(I) $Q_1\leq Q_2$, $\mathbb{P}$-a.s.\\[2mm]
(II) for any $i\in\{0,1,...,T-1\},$
$$f_1(i, Y^{(2)}_i, Z^{(2)}_i)\leq f_2(i, Y^{(2)}_i, Z^{(2)}_i),~~\mathbb{P}\text{-a.s.}$$
(III) there exist two constants $\omega_1,\omega_2>0$ such that the following holds: for each $i\in\{0,1,...,$ $T-1\},~y_1,y_2 \in \mathbb{R}$ and $z_1,z_2 \in
\mathbb{R}^{1\times((i+1)N)}$,
\begin{align*}
|f(i,y_1,z_1) - f(i, y_2, z_2)| \leq \omega_1 |y_1-y_2|
+ \omega_2\|z_1-z_2\|_{\bar{X}_i}.
\end{align*}
Moreover, $\omega_2$ satisfies
\begin{align*}
 6\omega_2^2 (Tr(C'C))^{\frac{1}{2}} Tr((\Psi_i^{\dagger})'\Psi_i^{\dagger})< 1.
\end{align*}
Then $Y^{(1)}_i\leq Y^{(2)}_i$, for any $i\in\{0,1,...,T-1\}$, $\mathbb{P}$-a.s.
\end{lemma}
Proof. This is an immediate application of Theorem 3.2 of Yang, Ramarimbahoaka and Elliott \cite{Yang}.
\section{A stochastic control problem}
We first describe the control problem. A feasible control $(u_m; m=0,1,...,T-1)$ is a
predictable process with values in a compact subset $U\subseteq\mathbb{R}^q$, where $q\in\mathbb{N}$. The set of feasible controls is
denoted by $\mathcal{U}$. Consider the following stochastic control problem: suppose the laws of the controlled process
belong to a family of equivalent measures whose densities are
\begin{align}\label{controlSDE}
\begin{cases}
V_m^u=1+ \sum\limits_{k=i+1}^{m}\alpha(k,u_k) V^u_k + \sum\limits_{k=i+1}^{m}V_{k}^u\beta(k,u_k)\Psi_k^{\dagger}\Psi_k(\Psi_k^{\dagger})'\bar{M}_k,\\
~~ ~~~ ~~~~~~ ~~~ ~~~~ ~~~~ ~~~~ ~~~~ ~~~ ~~~~ ~~~~ ~~~~ ~~~ ~~~~ ~~~~ ~~~~ ~~~ m \in \{i+1,...,T\};\\
V_i^u =1.
\end{cases}
\end{align}
Here the coefficients $\alpha(\omega, m, u) : \Omega\times\{0,1,...,T-1\}\times\mathbb{R}^q\rightarrow \mathbb{R}, \beta(\omega, m,u) :\Omega\times\{0,1,...,T-1\}\times \mathbb{R}^q\rightarrow\mathbb{R}^{1\times (N(m+1))}$
are predictable processes, uniformly bounded. Denote
the bounds of $|\alpha(m, u_m)|$ and $|\beta(m,u_m)|_N$ by $p$ and $l$, respectively.\\
\indent The control problem is to maximize over all feasible control processes $u_.\in\mathcal{U}$ the objective function
\begin{align*}
J(u_.) = E[\xi V^u_T +\sum_{k=0}^{T-1}g(k,u_k)V^u_k],
\end{align*}
where the terminal condition $\xi$ and the cost process $g(\cdot,\cdot)$ are $du\times P$-a.s. bounded in $\mathbb{R}$ and $V_.^u$ is the solution of SDE (\ref{controlSDE}). Set
$$f^u(k,y,z):=\alpha(k,u_k)y+ \beta(k,u_k) \Psi^{\dagger}_k\Psi_kz'+g(k,u_k),$$
for any $k=0,1,..., T-1,~y\in\mathbb{R},~z\in\mathbb{R}^{1\times (N(k+1))}$. Here $\Psi_.$ is given in equation (\ref{Psi}).
By Lemma \ref{Duality1}, we know $J(u_.) = Y_0^u$, where $Y^u_.$ is the solution to the following linear BSDE
\begin{align*}
\begin{cases}
Y_i^u= \xi + \sum\limits_{k=i}^{T-1} f^u(k,Y^u_k,Z_k^u) -  \sum\limits_{k=i}^{T-1}Z_{k}^u \bar{M}_{k+1},~~~~~~i\in\{0,1,...,T-1\};\\
Y_T^u= \xi.
\end{cases}
\end{align*}
\indent We now establish the main result, which describes the (sub)optional controls in terms of the Hamiltonian.
\begin{theorem} Set $f(k,y,z):=\textrm{esssup}\{f^u(k,y,z);~u_.\in\mathcal{U}\}$. Suppose $l$, which is the bound of $|\beta(i,u_i)|_N$, satisfies inequality (\ref{lconstant}). Further, suppose $l$ and $\lambda$ satisfy
\begin{equation}\label{landlambda}
 6l^2\lambda^2 (Tr(C'C))^{\frac{1}{2}} Tr((\Psi_i^{\dagger})'\Psi_i^{\dagger})< 1,
\end{equation}
where the constant $\lambda$ is given in Lemma \ref{normbound2}. Then BSDE
\begin{equation}\label{max}
\begin{cases}
Y_i= \xi + \sum\limits_{k=i}^{T-1} f(k,Y_k,Z_k) -  \sum\limits_{k=i}^{T-1}Z_{k} \bar{M}_{k+1},~~~i\in\{0,1,...,T-1\};\\
Y_T= \xi
\end{cases}
\end{equation}
has a unique solution $(Y_., Z_.)$. Moreover, set $Y_i^*= \textrm{esssup}\{Y_i^u;~u\in\mathcal{U}\},$ for each $i=0,1,...,T-1$. Then for each $i=0,1,...,T-1$,
$$Y_i = Y_i^*= \textrm{esssup}\{Y_i^u;~u\in\mathcal{U}\}.$$
\end{theorem}

\begin{proof}
By Lemma \ref{normbound2}, we know for any $k=0,1,...,T-1,~y,\tilde{y}\in\mathbb{R},~z,\tilde{z}\in\mathbb{R}^{1\times (N(k+1))}$,
\begin{align*}
&|f(k,y,z)-f(k,\tilde{y},\tilde{z})|\\
=&|\textrm{esssup}\{\alpha(k,u_k)(y-\tilde{y})+\beta(k,u_k) \Psi^{\dagger}_k\Psi_k(z-\tilde{z})';u_.\in\mathcal{U}\}|\\
\leq &p|y-\tilde{y}|+l|\Psi_k^{\dagger}\Psi_k(z-\tilde{z})'|_{N(k+1)}\\
\leq&p|y-\tilde{y}|+l\lambda\|(z-\tilde{z})'\|_{\bar{X}_k}.
\end{align*}
By Lemma \ref{exisunique} we deduce $(Y_.,Z_.)$ is the unique solution to BSDE (\ref{max}). For any $k=0,1,..., T-1$, set $\alpha_k
:=\textrm{ esssup}\{\alpha(k, u_k);~u \in \mathcal{U}\}$ and $\beta_k
:=\textrm{ esssup}\{\beta(k, u_k);$ $u \in \mathcal{U}\}.$
We now prove $Y = Y^*$. On the one hand, since $l$ and $\lambda$ satisfy inequality (\ref{landlambda}), by Lemma \ref{Comparison}, we know $Y_i\geq Y_i^u$, for any $i=0,1,...,T-1$, $u\in\mathcal{U}$. Hence $Y_i\geq Y_i^*$, for any $i=0,1,...,T-1$. On the other hand,  by the definition of $f$ , we know for any $\varepsilon> 0$, for each
$(\omega,i)\in\Omega\times\{0,1,...,T-1\},$
\begin{align*}\{u\in\mathcal{ U};&~f (\omega,i,Y_i(\omega),Z_i(\omega))\\
\leq&\alpha(\omega,i,u_i)Y_i(\omega) +\beta(\omega,i,u_i)\Psi_k^{\dagger}(\omega)\Psi_k(\omega) Z_i(\omega) + g(\omega, i,u_i) + \varepsilon\}
\end{align*}
is nonempty. Then using the Measurable Selection Theorem, for example, that can be found in Dellacherie \cite{Dellacherie} or in Bene\v{s} \cite{Bene70,Bene71}, there exists a $u^{\varepsilon}\in\mathcal{U}$ such that
\begin{align*}
f (i,Y_i,Z_i) \leq f^{u^{\varepsilon}}(i,Y_i,Z_i) + \varepsilon, ~~~i=0,1,...,T-1.
\end{align*}
Denote by $(Y_.^{u^{\varepsilon}},Z_.^{u^{\varepsilon}})$ the solution of the BSDE corresponding to $(f^{u^{\varepsilon}},\xi)$. Then, for any $i=0,1,...,T-1,$
\begin{align*}
&f^{u^{\varepsilon}}(i,Y^{u^{\varepsilon}}_i,Z^{u^{\varepsilon}}_i)- f (i,Y_i,Z_i)\\
\geq&f^{u^{\varepsilon}}(i,Y^{u^{\varepsilon}}_i,Z^{u^{\varepsilon}}_i)- f ^{u^{\varepsilon}}(i,Y_i,Z_i)-\varepsilon\\
=&\alpha(i,u^{\varepsilon}_i)(Y^{u^{\varepsilon}}_i-Y_i)+\beta(i,u^{\varepsilon}_i) \Psi^{\dagger}_i\Psi_i(Z^{u^{\varepsilon}}_i-Z_i)'-\varepsilon.
\end{align*}
For any $i=0,1,...,T-1,$ set
\begin{align*}&\eta_i\\[2mm]
:=&f^{u^{\varepsilon}}(i,Y^{u^{\varepsilon}}_i,Z^{u^{\varepsilon}}_i)- f(i,Y_i,Z_i)-\alpha(i,u^{\varepsilon}_i)(Y^{u^{\varepsilon}}_i-Y_i)\\[2mm]
&-\beta(i,u^{\varepsilon}_i) \Psi^{\dagger}_i\Psi_i(Z^{u^{\varepsilon}}_i-Z_i)'+\varepsilon.\end{align*}
 Thus, $\eta_i\geq0$, for $i=0,1,...,T-1.$ So for any $i=0,1,...,T-1,$ we deduce
\begin{align*}
&Y^{u^{\varepsilon}}_i-Y_i\\[2mm]
=&\sum\limits_{k=i}^{T-1} (\alpha(k,u^{\varepsilon}_k)(Y^{u^{\varepsilon}}_k-Y_k)+\beta(k,u^{\varepsilon}_k) \Psi^{\dagger}_k\Psi_k(Z^{u^{\varepsilon}}_k-Z_k)'-\varepsilon+\eta_i) \\[2mm]
&-  \sum\limits_{k=i}^{T-1} (Z^{u^{\varepsilon}}_k-Z_k)\bar{M}_{k+1}.
\end{align*}
By Lemma \ref{Duality1}, we have $Y^{u^{\varepsilon}}_i-Y_i=E[\sum\limits_{k=i}^{T-1} \eta_k\tilde{V}_k-\varepsilon\sum\limits_{k=i}^{T-1}\tilde{V}_k |\bar{X}_i]$, for $i=0,1,...,T-1,$ where $\tilde{V}_.$ is the solution of SDE
 \begin{align*}
\begin{cases}
\tilde{V}_m=1+ \sum\limits_{k=i+1}^{m}\alpha(k,u^{\varepsilon}_k) \tilde{V}_k + \sum\limits_{k=i+1}^{m}\tilde{V}_{k}\beta(k,u^{\varepsilon}_k)\Psi_k^{\dagger}\Psi_k(\Psi_k^{\dagger})'\bar{M}_k,\\
~~ ~~~~~~~~~ ~~~~~~~~~~~~~~~~~~~~~~~~~~~~~~~~~~~~~~~~~~~~~~ m \in \{t+1,...,T\};\\
\tilde{V}_i =1.
\end{cases}
\end{align*}
Since inequality (\ref{lconstant}) holds, by Lemma \ref{sde0} we know $\tilde{V}_m\geq0,$ for $m=0,1,...,T$. So
$Y^{u^{\varepsilon}}_i-Y_i\geq -\varepsilon E[\sum\limits_{k=i}^{T-1}\tilde{V}_k |\bar{X}_i]$, for $i=0,1,...,T-1$. Noting $Y^{u^{\varepsilon}}_i\leq Y_i$, by Lemma \ref{sdebound} we derive there exists a constant $\tilde{C}>0$ such that
$$
E[\max\limits_{i=0,1,...,T}|Y^{u^{\varepsilon}}_i-Y_i|^2]\leq T^2\varepsilon^2E[\max\limits_{k=0,1,...,T}|\tilde{V}_k|^2]\leq T^2\varepsilon^2\tilde{C}.
$$
Let $\varepsilon\rightarrow0$, we deduce $Y^{u^{\varepsilon}}_i\rightarrow Y_i$, for $i=0,1,...,T-1$. Thus, $Y_i=Y^*_i$, for $i=0,1,...,T-1$.
\end{proof}
\section{Conclusion} A stochastic control problem is considered in discrete time where the driving noise is determined by a semi-Markov chain and the dynamics involve a control parameter. A verification result is derived in terms of a Hamiltonian. 

\section*{Acknowledgements}
The authors thank the Natural Sciences and Engineering Research Council of Canada for continuing support.

\end{document}